\documentclass[12pt,letter]{amsart}

\usepackage{amssymb,latexsym, amsmath, amsxtra, mathrsfs, bm}

\usepackage[dvips]{graphics}
\usepackage{xypic,xcolor}
\usepackage{subcaption}

\usepackage{verbatim}
\usepackage[abs]{overpic}
\usepackage{hyperref}
\usepackage{mathtools}

\usepackage{tikz}           
\usepgflibrary{arrows.meta}                 
\usetikzlibrary{decorations.pathreplacing}  
\usepackage{tikz-cd}        

\DeclareMathAlphabet{\mathbbold}{U}{bbold}{m}{n}

\allowdisplaybreaks[3]

\theoremstyle{plain}
        \newtheorem{theorem}{Theorem}[section]
        \newtheorem*{theorem*}{Theorem}
        \newtheorem*{conj*}{Conjecture}
        \newtheorem{lemma}[theorem]{Lemma}
        \newtheorem{prop}[theorem]{Proposition}

\theoremstyle{definition}
        
        \newtheorem{rem}[theorem]{Remark}

\theoremstyle{remark}
        \newtheorem*{remark}{Remark}

\numberwithin{equation}{section}
\numberwithin{theorem}{section}
\numberwithin{table}{section}
\numberwithin{figure}{section}

\providecommand{\defn}[1]{\emph{#1}}

\renewcommand{\le}{\leqslant}
\renewcommand{\leq}{\leqslant}
\renewcommand{\ge}{\geqslant}
\renewcommand{\geq}{\geqslant}

\newcommand{\var}  {\operatorname{var}}

\newcommand{\card} {\operatorname{card}}

\newcommand{\R}{\mathbb{R}}

\newcommand{\N}{\mathbb{N}}

\providecommand{\abs}[1]{\lvert#1\rvert}

\providecommand{\norm}[1]{\|#1\|}

\renewcommand{\:}{\colon}

\newcommand{\LIP}{\operatorname{LIP}}

\newcommand{\lip}{\operatorname{lip}}

\newcommand{\E} {\mathbf{E}}

\newcommand{\CCC}{C}

\newcommand{\MMM}{\mathcal{M}}

\newcommand{\Holder}[1] {\CCC^{0,#1}}
\newcommand{\holder}[1] {c^{0,#1}}

\newcommand{\vertiii}[1]{{\left\vert\kern-0.25ex\left\vert\kern-0.25ex\left\vert #1
    \right\vert\kern-0.25ex\right\vert\kern-0.25ex\right\vert}}

\renewcommand{\=}{\coloneqq}

\newcommand{\head}{\operatorname{head}}

\newcommand{\tail}{\operatorname{tail}}

\newcommand{\len}{\operatorname{len}}

\newcommand{\BG}{\operatorname{\mathcal{BG}}}

\newcommand{\Gap}{\operatorname{Gap}}

\renewcommand{\E}{\mathbb{E}}

\renewcommand{\P}{\mathbb{P}}

\newcommand{\Lip}{\operatorname{Lip}}

\newcommand{\Lock}{\operatorname{Lock}}
\newcommand{\lock}{\operatorname{lock}}

\newcommand{\disagree}{\mathbin{\dagger}}

\newcommand{\ba}{\bm{a}}
\newcommand{\bb}{\bm{b}}

\newcommand{\be}{\bm{e}}
\newcommand{\bu}{\bm{u}}

\newcommand{\bw}{\bm{w}}
\newcommand{\bx}{\bm{x}}
\newcommand{\by}{\bm{y}}
\newcommand{\bz}{\bm{z}}

\newcommand{\cC}{\mathcal{C}}
\renewcommand{\cD}{\mathcal{D}}
\newcommand{\cE}{\mathcal{E}}

\renewcommand{\cH}{\mathcal{H}}

\newcommand{\cM}{\mathcal{M}}

\newcommand{\cQ}{\mathcal{Q}}
\renewcommand{\cR}{\mathcal{R}}
\newcommand{\cS}{\mathcal{S}}

\newcommand{\sC}{\mathscr{C}}

\newcommand{\sP}{\mathscr{P}}

\newcommand{\sU}{\mathscr{U}}

\newcommand{\ob}{\overline{b}}

\newcommand{\ow}{\overline{w}}

\newcommand{\tQ}{\widetilde{Q}}

\newcommand{\tphi}{\widetilde{\phi}}

\begin{document}
\title[Prevalence of periodicity of maximizing measures]{On the prevalence of the periodicity of maximizing measures}
\author{Jian~Ding \and Zhiqiang~Li \and Yiwei~Zhang}
\address{Jian~Ding, School of Mathematical Sciences, Peking University, Beijing 100871, China}
\email{dingjian@math.pku.edu.cn}
\address{Zhiqiang~Li, School of Mathematical Sciences \& Beijing International Center for Mathematical Research, Peking University, Beijing 100871, China}
\email{zli@math.pku.edu.cn}
\address{Yiwei~Zhang, Department of Mathematics \& SUSTech International Center for Mathematics, Southern University of Science and Technology, Shenzhen, Guangdong 518055, China}
\email{zhangyw@sustech.edu.cn}




\begin{abstract}
For a continuous map $T\: X\rightarrow X$ on a compact metric space $(X,d)$, we say that a function $f\: X \rightarrow \R$ has the property $\sP_T$ if its time averages along forward orbits of $T$ are maximized at a periodic orbit. In this paper, we prove that for the one-sided full shift on two symbols, the property $\sP_T$ is prevalent (in the sense of Hunt--Sauer--Yorke) in spaces of Lipschitz functions with respect to metrics with mildly fast decaying rate on the diameters of cylinder sets.  This result is a strengthening of \cite[Theorem~A]{BZ16}, confirms the prediction mentioned in the ICM proceeding contribution of J.~Bochi (\cite[Section~1]{Boc18}) suggested by experimental evidence, and is another step towards the Hunt--Ott conjectures in the area of ergodic optimization.
\end{abstract}

\maketitle

\tableofcontents

\section{Introduction}     \label{s:Introduction}

Mathematicians have hoped for an analog of the notions of ``Lebesgue almost every'' and ``Lebesgue measure zero'' in infinite-dimensional spaces such as various Banach spaces of functions studied in dynamics. Such a desire was recorded in the contribution of A.~N.~Kolmogorov to the 1954 International Congress of Mathematicians \cite{Ko54}.

A natural notion to fulfill this quest called \emph{prevalence} was introduced by B.~R.~Hunt, T.~Sauer, and J.~A.~Yorke \cite{HSY92}. This notion has since proven to play a central role in understanding generic behaviors in dynamics from a probabilistic (measure-theoretic) perspective. For more history, related notions, and applications in analysis, dynamics, economics, etc., we refer the readers to the surveys by B.~R.~Hunt and V.~Yu.~Kaloshin \cite{HK10} and by W.~Ott and J.~A.~Yorke \cite{OY05}.

Following \cite{HK10}, we recall the notion of prevalence as follows.

Let $V$ be a topological vector space over $\R$ equipped with a complete metric. A Borel measure $\mu$ on $V$ is called \defn{transverse} to a Borel set $B\subseteq V$ if $\mu( \{ x + v  :  x \in B \} ) = 0$ for every $v\in V$. We say that a Borel subset $B$ of $V$ is \defn{shy} if there exists a compactly supported Borel probability measure transverse to $B$. A subset $S$ of $V$ is called \emph{prevalent} (resp.\ \defn{shy}) if $V\setminus S$ (resp.\ $S$) is contained in a shy Borel subset of $V$.


Prevalence has the following properties (see for example, \cite[Section~2.1]{HK10}): it is preserved under translation and countable intersections, it implies density, and it coincides with the notion of having full Lebesgue measure when the ambient space is a finite-dimensional Euclidean space.

The aim of this paper is to investigate the prevalence of some ergodic properties in the study of dynamical systems.


Let $T\: X \rightarrow X$ be a continuous transformation. Let $f \: X \rightarrow \R$ be a continuous function, called a potential. We denote the maximal potential energy by 
\begin{equation*}\label{e:ergodicmax}
  \beta(f)\=\sup\limits_{\mu\in\MMM(X,T)}\int\! f \, \mathrm{d}\mu=\max\limits_{\mu\in\MMM(X,T)}\int\! f \, \mathrm{d}\mu.
\end{equation*}
The last identity follows from the weak*-compactness of the set $\MMM(X,T)$ of $T$-invariant Borel probability measures on $X$. The quantity $\beta(f)$ can also be expressed in terms of the maximal time average of $f$ (see for example, \cite[Proposition~2.1]{Je06}). We call a measure $\mu \in \MMM(X,T)$ that maximizes the potential energy $\int\!f \,\mathrm{d}\mu$ a \emph{measure of maximal potential energy} with respect to $f$ or a \emph{$f$-potential-energy-maximizing measure} (or a \emph{$f$-maximizing measure}). 

If one of the $f$-maximizing measures is supported on a periodic orbit of $T$, then we say that $f$ has \emph{property $\sP_T$} and call such a measure \emph{periodic}. We define the following subsets of the set $\CCC(X)$ of real-valued continuous functions:
\begin{align*}
\sP(T) & \= \{ f \in \CCC(X) : f \text{ has property } \sP_T \},\\
\sP\sU(T) & \= \{ f \in  \sP(T) : f \text{ has a unique $f$-maximizing mesure} \}.
\end{align*}

The set of functions with property $\sP_T$ has been investigated in various Banach spaces of continuous functions for many dynamical systems, see for example, \cite{Bou00, Bou01, CLT01, QS12, Co16, BZ16, HLMXZ19, LZ23}. The study of typical properties of maximizing measures for Lipschitz and H\"older functions has been particularly fruitful.

Assume that $(X,d)$ is a compact metric space with infinite cardinality. Let $\Lip(X,d)$ be the space of real-valued Lipschitz functions on $(X,d)$. A function $f \: X \rightarrow \R$ is a \emph{little Lipschitz function} or \emph{locally flat Lipschitz function} if
\begin{equation*}
	\sup\{ \abs{ f(x) - f(y) } : x,\,y \in X, \, d(x,y) \leq r \} = o(r) \qquad \text{ as } r \to 0.
\end{equation*}
The space of real-valued little Lipschitz functions on $(X,d)$ is denoted by $\lip(X,d)$ and called the \emph{little Lipschitz space}. We equip both spaces with the \emph{Lipschitz norm}
\begin{equation*}
	\norm{  f  }_{\Lip} \= \norm{ f }_{\infty} + \LIP_d(  f ) ,
\end{equation*}
i.e., the sum of the sup-norm and the minimal Lipschitz constant $\LIP_d(  f )$ of $f$ given by
$\LIP_d (f) \= \sup\{ \abs{f(x) - f(y)} / d(x,y)  :  x,\,y \in X,\, x \neq y \}$.
The snowflake $d^\alpha$ of $d$ given by $d^\alpha(x,y) \= d(x,y)^\alpha$ is also a metric for each $\alpha \in (0,1]$. The space $\Holder{\alpha}(X,d)$ of $\alpha$-H\"older continuous functions is precisely $\Lip(X,d^\alpha)$. The space $\holder{\alpha}(X,d) \= \lip (X,d^\alpha)$ is called the space of \emph{little $\alpha$-H\"older functions} by some authors.

We say that a function $f \in \Lip(X,d)$ (resp.\ $\lip(X,d)$) has the \emph{locking\footnote{This is translated from the term \emph{verrouillage} used by T.~Bousch in \cite[Secton~8]{Bou00}.} property} if it has a unique maximizing measure $\nu \in \MMM(X,T)$ that is supported on a periodic orbit of $T$ (in particular, $f$ has property~$\sP_T$) and moreover $\nu$ is the unique maximizing measure for every function in $\Lip(X,d)$ (resp.\ $\lip(X,d)$) sufficiently close to $f$ in the Lipschitz norm. The set of Lipschitz functions $f \in \Lip(X,d)$ (resp.\ $\lip(X,d)$) with the locking property is denoted by $\Lock_T (X,d)$ (resp.\ $\lock_T (X,d)$). The locking property clearly depends on the metric $d$.

The set of Lipschitz functions with the locking property is exactly the interior of the set of Lipschitz functions with property~$\sP_T$. One inclusion is obvious from the definition, and the other was shown by G.~Yuan and B.~R.~Hunt \cite[Remark~4.5]{YH99} (see also \cite{BZ15}). The same property holds for little Lipschitz functions, see \cite[Theorem~4.1]{LZ23}. In particular, the set of functions with the locking property is dense among the functions with property~$\sP_T$.

Consider the full shift $(\Sigma^+, \sigma )$ on the set $\Sigma^+ \= \{0,\,1\}^{\N}$ of binary sequences. Here the \emph{left-shift operator} $\sigma  \: \Sigma^+ \rightarrow \Sigma^+$ is given by
\begin{equation*}
\sigma  ( \{w_i\}_{i\in\N} ) = \{w_{i+1}\}_{i\in\N}  \qquad \text{for } \{w_i\}_{i\in\N} \in \Sigma^+.
\end{equation*}

For a strictly decreasing sequence $\ba = \{ a_n \}_{n\in\N_0}$ of positive numbers converging to zero, define a function $d_{\ba} \: \Sigma^+ \times \Sigma^+ \rightarrow \R$ by assigning, for each pair of distinct $\bw= \{w_n\}_{n\in\N}, \, \bz=\{z_n\}_{n\in\N} \in \Sigma^+$, 
\begin{equation}   \label{e:def_d_a}
d_{\ba} ( \bw, \bz ) \= a_m 
\end{equation}
where $m$ is the smallest positive integer with the property that $w_n \neq z_n$, and by setting $d_{\ba}  ( \bz, \bz ) \= 0$ for each $\bz \in \Sigma^+$. It is easy to see that $d_{\ba}$ is a metric on $\Sigma^+$.

We consider the  Banach spaces $\Lip(\Sigma^{+},d_{\ba})$ and $\lip(\Sigma^{+},d_{\ba})$ with respect to the metric $d_{\ba}$. Observe that the faster the sequence ${\ba}$ decays, the more regular  the functions in $\Lip(\Sigma^{+}, d_{\ba})$ are.

We are now ready to state our main theorem below.

\begin{theorem}[Main Theorem]   \label{t:prevalence}
Let $(\Sigma^+, \sigma )$ be the full shift on binary sequences.
Let $\theta$ be a number in $(0,1/4)$ and $\ba=\{a_{n}\}_{n\in\N_0}$ be a strictly decreasing sequence of positive numbers satisfying 
\begin{equation}   \label{e:decay}
\frac{ a_{n+1} }{ a_n } = O ( \theta^n ) \qquad \text{ as } n \to +\infty.
\end{equation}
Then $\sP\sU(\sigma) \cap \Lip(\Sigma^{+},d_{\ba})$ is a prevalent set in the space $\Lip(\Sigma^{+},d_{\ba})$ of real-valued Lipschitz functions, where $\sP\sU(\sigma)$ consists of continuous functions $f \: \Sigma^{+} \rightarrow \R$ with the property that there is a unique $f$-maximizing measure $\mu$ and that $\mu$ is supported on a periodic orbit of $\sigma$. Moreover, $\Lock_{\sigma}(\Sigma^{+},d_{\ba})$ is a prevalent set in $\Lip(\Sigma^{+},d_{\ba})$.
\end{theorem}

In general, prevalence does not guarantee topological genericity, neither does the latter imply the former. However, due to the fact that $\Lock_{\sigma}(\Sigma^{+},d_{\ba})$ is known to be an open set in $\Lip(\Sigma^+,d_{\ba})$, topological genericity can be guaranteed by the density alone, which is always a direct consequence of prevalence. Therefore, Theorem \ref{t:prevalence} represents, albeit in our special setting, a strengthening of the kind of topological result obtained in \cite[Theorem~A]{Co16}.

As explained in \cite[Section~1]{Boc18}, the type of results like Theorem~\ref{t:prevalence} provide some confirmation of the experimental findings of B.~R.~Hunt and E.~Ott published over two decades ago \cite{HO96a,HO96b}. 

In the late 1990s, B.~R.~Hunt, E.~Ott, and G.~Yuan conjectured that for a typical chaotic system $T$, a typical smooth function $f$ admits a maximizing measure supported on a periodic orbit, where the typicality on the function $f$ was both in the sense of probability (in terms of the Lebesgue measure of the parameter(s) for the potential $f$ in a parameter space) (\cite[Conjecture]{HO96a} and \cite[Conjecture~2]{HO96b}) and in the sense of topology (topological genericity) (\cite[Conjecture~1.1]{YH99}). We recall that such a function $f$ has property $\sP_T$. Moreover, B.~R.~Hunt and E.~Ott observed based on numerical evidence that the set of the functions $f$ without property $\sP_T$ has fractal dimension zero. Despite the recent breakthroughs in the direction of topological genericity initiated by \cite{Co16}, the Hunt--Ott conjecture has been much less thoroughly investigated.

For a specific map $T(x)=2x$ on the circle, and a one-parameter family $\rho_\tau(x)=\cos2\pi(x-\tau)$, Corollary~1 in \cite{Bou00} shows that the maximizing measure is unique and that it is supported on a periodic orbit for every $\tau$ except on a set which has zero Lebesgue measure and zero Hausdorff dimension  (see also \cite{Je00}). Therefore, the Hunt--Ott conjecture is confirmed in this specific one-parameter family. The notion of full Lebesgue measure coincides in finite dimension with the notion of prevalence in Theorem~\ref{t:prevalence}. 

In fact, Theorem~\ref{t:prevalence} is a strengthened version of \cite[Theorem A]{BZ16} by J.~Bochi and the third-named author, where they assume the \emph{evanescent} condition on $\ba$, i.e., $\ba=\{a_{n}\}_{n\in\N_0}$ is a strictly decreasing sequence of positive numbers satisfying
\begin{equation}\label{e:eva}
\frac{a_{n+1}}{a_n} = O \bigl( 2^{-2^{n+2}} \bigr)  \qquad \text{ as } n \to +\infty .
\end{equation}
Theorem~\ref{t:prevalence} guarantees the prevalence of continuous functions with the locking property in bigger subspaces of H\"{o}lder continuous functions (than the space of functions satisfying \eqref{e:eva}), which confirms the prediction mentioned in \cite[Section~1]{Boc18} suggested by experimental evidence, and is another step towards the Hunt--Ott conjectures.

As a counterpart to Theorem~\ref{t:prevalence}, the following result for little Lipschitz functions holds.

\begin{theorem}   \label{t:prevalence_lip}
Assume the hypotheses of Theorem~\ref{t:prevalence}. Then $\sP\sU(\sigma) \cap \lip(\Sigma^{+},d_{\ba})$ is a prevalent set in $\lip(\Sigma^{+},d_{\ba})$. Moreover, $\lock_{\sigma}(\Sigma^{+},d_{\ba})$ is a prevalent set in $\lip(\Sigma^{+},d_{\ba})$.
\end{theorem}



%
%
%
%

\begin{rem}
Since $d_{\ba}^\alpha = d_{\ba^\alpha}$ where $\ba^\alpha \= \{a_{n}^\alpha\}_{n\in\N_0}$, it is easy to see that we can replace the metric $d_{\ba}$ by $d_{\ba}^{\alpha}$ for $\alpha \in  ( - \log_\theta 4, 1]$ in Theorems~\ref{t:prevalence} and~\ref{t:prevalence_lip} to extend these results to H\"older continuous functions with exponent sufficiently close to $1$.
\end{rem}

\smallskip

\textbf{Main ideas of the proof.} The strategy of the proof of Theorem~\ref{t:prevalence} relies on an approach of finite-dimensional approximations via Haar wavelets for ergodic optimization, which can be conveniently re-expressed as a maximum mean cycle problem on de~Bruijn--Good directed graphs. Due to the hypothesis~\eqref{e:decay}, the coefficients in the Haar series for a Lipschitz function (with respect to $d_{\ba}$) decay fast in a controlled way. Our main goal is to show that with high probability, the maximizing measure for a function $f$ coincides with that for a step function obtained from truncating the Haar series for $f$ at some finite step in the finite-dimensional approximations. To this end, we adapt some idea from \cite[Section 5]{CCK15}, and make refined estimates to get a lower bound for the fluctuation on the (conditional) gap between the mean-weights of the two heaviest directed cycles on the sequence of de~Bruijn--Good directed graphs.  These estimates ensure that the hypothesis~\eqref{e:decay} is sufficient to guarantee the prevalence of the periodicity of the maximizing measure. For a comparison between our strategy and that from \cite{CCK15}, see a discussion before the proof of Proposition~\ref{p:conditional_layered_Gap_estimate} in Section~\ref{s:random_perturbation}.

De~Bruijn--Good directed graphs were independently introduced by N.~G.~de~Bruijn \cite{deBr46} and I.~J.~Good \cite{Go46} in 1946, and have played important roles in both pure mathematics (for example, enumerative combinatorics) and applied mathematics (for example, genome assembly). The appearance of de~Bruijn--Good directed graphs in this paper is natural, as walks on these graphs correspond to orbits under the shift operator $\sigma$ (see \cite{BZ16}).

Investigations on the directed cycle with the maximum mean-weight in a directed graph are sometimes known as the \emph{maximum mean cycle problem}. This problem is one of the most fundamental problems in combinatorial optimization and has a long history of algorithmic studies, for example, \cite{DIG99, Ka78}.

It is also worth mentioning that our methods should be extendable to shifts on an arbitrary finite number of symbols or more generally to subshifts of finite type by investigating natural generalizations of de~Bruijn--Good directed graphs. Favoring clarity at the expense of generality, we opted to work with the full shift on $2$ symbols. In this context, we are able to express with relative ease the random perturbations key to the proof of Theorem~\ref{t:prevalence}.

\smallskip

To compare Theorem~\ref{t:prevalence} with \cite[Theorems~A and B]{BZ16}, we note that the proof of Theorem~B (and consequently Theorem~A) in \cite{BZ16} bounds the probability $\P_{\textbf{b}}\{g\in\mathcal{H}_{\textbf{b}} : \Gap_n(f_0+g)\leq \epsilon_n\}$ by the maximum number of pairs of neighbours in $\sC(\BG_n)$ (see \cite[Equation~(5.3)]{BZ16}). By \cite[Subsection~3.5]{BZ16}, such a number is no more than $2^{2^{n}}$. Hence the approach in \cite{BZ16} requires the evanescent condition to offset the impact of $2^{2^{n}}$. 

In contrast, Proposition~\ref{p:conditional_layered_Gap_estimate} in this paper bounds a similar probability with more sophisticated probabilistic tools, leading to an improvement from the evanescence condition to condition~(\ref{e:decay}) in the current paper. Two main new ingredients are discussed as follows. First, in Propositions~\ref{p:conditional_layered_Gap_estimate} and~\ref{p:goal}, we obtain the desired bound on the probability of small gap from a bound on some conditional probability  $\mathbb{P}(\{\Gap_{n,\bm{t}}\leq \epsilon_{n}\})$ of small gap. Here $\Gap_{n,\bm{t}}$ is the absolute value of the difference between the largest two values in $\{ X_{\bm{t}} (\cC)  :  \cC \in \sC(\BG_n) \}$, where $\{ X_{\bm{t}} (\cC) : \cC \in \sC (\BG_n) \}$ has the distribution of the collection of random variables $\bigl\{ M_n^{f_0 +g}(\cC) : \cC \in \sC (\BG_n) \bigr\}$ under the conditioning of $\{Y_{\bw} = t_{\bw} : \bw \in \Sigma_{\leq n-2} \} $. Second, by bounding the probability density function of $X_{\bm{t}} (\cC) $ and applying the concentration inequality (for the uniform distribution), we get an upper bound for $\P(\{\Gap_{n,\bm{t}}\leq \epsilon_{n}\})$ under an assumption on exponential decay of $\epsilon_{n}$ (rather than the double-exponential decay in \cite{BZ16}).

In view of the connection to the coding of hyperbolic smooth systems, it is natural to consider the case where $a_n=O (\theta^{n} )$, for $\theta\in(0,1)$. In such a case, $d_{\ba}$ coincides with the classical Cantor distance on $\Sigma^{+}$. However, a major obstacle for proving Theorem~\ref{t:prevalence} in such a case lies in obtaining an effective lower bound for $\P(\{\Gap_{n,\bm{t}} \leq \epsilon_{n}\})$ when $\epsilon_{n}$ decays to zero at a speed slower than an exponential rate (c.f.\ Proposition~\ref{p:conditional_layered_Gap_estimate}).

In addition, the upper bound $1/4$ on $\theta$ in Theorem~\ref{t:prevalence} comes from the constants $2^{-n}$ and $2^{n}$ in (\ref{e:conditional_layered_Gap_estimate}) from Proposition~\ref{p:conditional_layered_Gap_estimate}. These constants come from our choices of $\{Y_{\bw}\}_{\bw\in\Sigma_*}$ as i.i.d.~random variables with the uniform distribution on $[-1,1]$. We suspect that the constant $1/4$ might be improvable by using other distributions for $\{Y_{\bw}\}_{\bw\in\Sigma_*}$. However, the alternative proof could be significantly lengthier and too complicated to showcase the main ideas of our strategy. Given that the case where $a_n=O(\theta^{n})$ for $\theta\in(0,1)$ remains open, such technical considerations may be better suited for subsequent investigations.

\smallskip

The paper is organized as follows. Section~\ref{s:Preliminaries} is devoted to preliminaries, where we recall some basic facts about full shifts, Haar series, and the compactness of Hilbert bricks (see Lemma~\ref{l:brick_def}). Section~\ref{s:concrete_formulation} provides a more concrete formulation (see Theorem~\ref{t:brick}) of our main result Theorem~\ref{t:prevalence}, with the proof of the implication included at the end of the section. In Section~\ref{s:Graph_gap} we recall the de~Bruijn--Good directed graphs and the gap criterion for the locking property that was developed in \cite{BZ16}, see Lemma~\ref{l:gap}. In Section~\ref{s:random_perturbation} we adapt some idea from \cite[Section~5]{CCK15} to make refined estimates for the (conditional) gap between the mean-weights of the two heaviest directed cycles, see Proposition~\ref{p:conditional_layered_Gap_estimate} and Proposition~\ref{p:goal}. Using these estimates, we show in Section~\ref{s:Proof_T_Brick} that hypothesis~\eqref{e:decay} guarantees that with probability $1$, the gap criterion is satisfied at some finite step in the finite-dimensional approximations, establishing Theorem~\ref{t:brick}. Thus we complete the proof of Theorem~\ref{t:prevalence}. Finally, Section~ \ref{s:gound_states} is devoted to the proof of Theorem~\ref{t:prevalence_lip}.

\subsection*{Acknowledgments} 
The authors want to thank Jairo~Bochi and Oliver Jenkinson for encouragement, and the anonymous referee(s) for valuable comments. Y.~Zhang is grateful to Peking University for the hospitality, during his stay when part of this work was done. J.~Ding was partially supported by NSFC Key Program Project No. 12231002. Z.~Li was partially supported by NSFC Nos.~12101017, 12090010, 12090015, and BJNSF No.~1214021. Y.~Zhang was partially supported by NSFC Nos. 12161141002, 12271432, and Guangdong Basic and Applied Basic Research Foundation.

\section{Preliminaries}  \label{s:Preliminaries}

In this section, we will introduce the background and necessary concepts.

\subsection{Full shift} \label{ss:Full_shift}
We follow the convention that $\N \coloneqq \{1, \, 2, \, 3, \, \dots\}$ and $\N_0 \coloneqq \{0\} \cup \N$ in this paper.

For each $n \in \N$, the set $\Sigma_{n} \= \{0,\,1\}^{n}$ of \emph{words of length $n$} (or \emph{$n$-words}) in the alphabet $\{0,\,1\}$ is the set of binary sequences of length $n$. Write $\Sigma_{0} \=\{\emptyset \}$ and call $\emptyset$ the word of length $0$. Define $\abs{\bw} \=n$ for each $n$-word, $n\in\N_0$. The set $\Sigma^+ \= \{0,\,1\}^{\N}$ of \emph{infinite words} in the alphabet $\{0,\,1\}$ is the set of one-sided infinite binary sequences.  Denote the set of \emph{finite words} by $\Sigma_* \= \bigcup_{n \in \N_0 }\Sigma_{n}$ and the set of words of length at most $k$, for $k\in\N_0$, by $\Sigma_{\leq k} \= \bigcup_{n = 0}^k \Sigma_{n}$. Note that $\Sigma_*$ includes $\emptyset$, i.e., the word of length $0$. We often write $\{w_n\}_{n\in\N} = w_1 w_2 \dots w_n \dots$ for an infinite word $\bw= \{w_n\}_{n\in\N} \in \Sigma^+$, and similarly, write $\{z_n\}_{n=1}^m = z_1 z_2 \dots z_m$ for an $m$-word $\bz= \{z_n\}_{n=1}^m \in \Sigma_m$, $m\in\N$.

The \emph{left-shift operator} $\sigma  \: \Sigma^+ \rightarrow \Sigma^+$ is given by
\begin{equation*}
\sigma  ( \{w_i\}_{i\in\N} ) = \{w_{i+1}\}_{i\in\N}  \qquad \text{for } \{w_i\}_{i\in\N} \in \Sigma^+.
\end{equation*}
The pair $(\Sigma^+, \sigma )$ is called the \defn{full shift} on $\Sigma^+ = \{0,\,1\}^{\N}$.

For each pair of finite words $\bw, \, \bz\in\Sigma_*$, define $\bw\bz$ to be their concatenation if $\bw\neq \emptyset$ and set $\bw\bz \= \bz$ otherwise, and denote by $[\bw] \=\{\bw\by  :  \by \in \Sigma^+\}$ the \emph{cylinder set} beginning at $\bw$.

An infinite word $\bw \in \Sigma^+$ is a \emph{periodic point} of $(\Sigma^+, \sigma )$ if $\sigma^n(\bw) = \bw$ for some $n \in \N$, and the smallest such positive integer $n$ is called the \emph{period} of $\bw$. Let $P_\sigma$ denote the set of all periodic points of $(\Sigma^+, \sigma )$.

Let $\bx = \{x_i\}_{i\in\N_0}$ and $\by = \{y_i\}_{i\in\N_0}$ be two elements of $\Sigma^{+}$. Denote by $\bx \disagree \by$ the position of first disagreement between the sequences $\bx$ and $\by$, that is, the least $m \in \N$ such that $x_m \neq y_m$, with the convention $\bx \disagree \bx = +\infty$.
The following properties hold:
\begin{equation*}
\bx \disagree \by = \by \disagree \bx , \qquad
\bx \disagree \bz \ge \min\{ \bx \disagree \by, \, \by \disagree \bz \}.
\end{equation*}
For each $n \in \N$,
we define the \emph{$n$-th variation} of a function $f \: \Sigma^{+} \rightarrow \R$ as:
\begin{equation}\label{e:var}
\mbox{var}_n(f) \= \sup_{\bx\disagree  \by \ge n} \abs{ f(\bx)-f(\by) }  .
\end{equation}

\subsection{Haar series}  \label{ss:Haar}
As mentioned in \cite[Section 1.3]{BZ16}, for each $\bw \in \Sigma_*$, 
we define the \emph{Haar function}:
\begin{equation}\label{e:Haar_function}
h_{\bw}\=\frac{\mathbbold{1}_{[\bw0]}-\mathbbold{1}_{[\bw1]}}{2}.
\end{equation}
These functions are straightforward adaptations of the classical Haar functions on $[0,1]$ (see for example, \cite[Section~6.3]{Pi09}). It has an advantage to work on the Cantor set $\Sigma^{+}$ over working on the interval $[0,1]$: the Haar functions defined in \eqref{e:Haar_function} are continuous. We follow the choice of the normalization in \cite{BZ16}, which makes subsequent formulas simpler.

If a function $f \: \Sigma^{+} \rightarrow \R$ satisfies $\mbox{var}_n(f) = 0$ for some $n\in\N_0$, i.e., it is constant on each cylinder of length $n$, then it is called a \emph{step function of level $n$}. Such functions form a vector space $\cS_n$ of dimension $2^n$. For each $n \in \N$, the set $\{1\} \cup \{h_{\bw}  :   \abs{\bw}<n\}$ forms a basis of the vector space $\cS_n$.

Let $\beta$ be the unbiased Bernoulli measure on $\Sigma^{+}$, i.e., the probability measure that assigns equal weights to all cylinders of the same length. Let $L^2(\beta)$ denote the Hilbert space of functions that are square-integrable with respect to the measure $\beta$. Then the set $\{1\} \cup \{h_{\bw}  :  \bw \in \Sigma_*\}$ is an orthogonal basis of $L^2(\beta)$. Thus every $f \in L^2(\beta)$ can be represented by a \emph{Haar series}:
\begin{equation*}
f = c_{\emptyset} (f) + \sum_{{\bw} \in \Sigma_*} c_{\bw}(f) h_{\bw}  \qquad \text{(equality in $L^2(\beta)$),}
\end{equation*}
where the \emph{Haar coefficients} are defined as:
\begin{equation}\label{e:Haar_coefficients}
c_{\emptyset} (f) \= \int\! f \,\mathrm{d}{\beta}   , \qquad
c_{\bw}(f) \= 2^{\abs{\bw}+2} \int\! f \, h_{\bw} \,\mathrm{d}\beta,
\end{equation}
for each $\bw\in\Sigma_{n}$. It follows from \cite[Equations (2.2), (2.3), and (2.4)]{BZ16} that the \emph{$n$-th approximation} $A_n f$ of $f \in L^2(\beta)$ has the following equivalent characterizations:
\begin{itemize}
\item the projection of $f$ on the subspace $\cS_n$ along the orthogonal complement of $\cS_n$;
\item the sum of the truncated Haar series:
\begin{equation}  \label{e:A_truncated}
A_n f = c_{\emptyset} (f) + \sum_{\abs{\bw}<n} c_{\bw}(f) h_{\bw}  ;
\end{equation}
\item the function obtained by averaging $f$ on cylinders of length $n$:
\begin{equation}  \label{e:A_average}
A_n f = \sum_{ \bw \in \Sigma_n } \left( 2^n \int_{[\bw]} \! f \, \mathrm{d}\beta \right) \mathbbold{1}_{[\bw]}   .
\end{equation}
\end{itemize}
Thus we have
\begin{equation*}
\norm{f - A_n f}_\infty \le \mbox{var}_n(f)  ,
\end{equation*}
which, under the additional assumption that $f$ is continuous, converges to $0$ as $n$ tends to $+\infty$. Therefore every continuous function $f$ can be written as a uniformly convergent series:
\begin{equation*}
f = c_{\emptyset} (f) + \sum_{n=0}^{+\infty}   \sum_{ {\bw} \in \Sigma_n } c_{\bw}(f) h_{\bw}  ,
\end{equation*}
which, with some abuse of language, we also call a \emph{Haar series}.

\subsection{Gauges}  \label{ss:gauge}

Let $\ba = \{a_n\}_{n \in \N_0}$ be a strictly decreasing sequence of positive numbers converging to $0$. 

Consider a function $f \in \CCC (\Sigma^+)$. It follows from \eqref{e:var} and \eqref{e:Haar_coefficients} that for each $n\in\N$ and each $\bw \in \Sigma_n$, we have
\begin{equation}  \label{e:c_vs_var}
c_{\bw} (f) \leq   \var_n (f).
\end{equation}
Similarly, if $f$ is Lipschitz with respect to $d_{\ba}$, then
\begin{equation}  \label{e:LIP_vs_var_over_an}
\LIP_{d_{\ba}} (f) = \sup_{ n \in \N }  \frac{ \var_n(f) }{ a_n }.
\end{equation}

A \emph{gauge} is a family $\bb = \{b_{\bw}\}_{\bw \in \Sigma_*}$ of positive numbers indexed by finite words. We say that a gauge $\bb = \{b_{\bw}\}_{\bw \in \Sigma_*}$ is weakly \emph{$\ba$-admissible} if
\begin{equation}\label{e:adm}
\ob_n  = o(a_n)
\quad \text{as } n \to +\infty,
\end{equation}
where
\begin{equation*}
\ob_n\= \max_{ \bw \in \Sigma_n } \{ b_{\bw} \}
\quad \text{for each } n\in\N_0 .
\end{equation*}

The proof of \cite[Lemma~1.1]{BZ16} in \cite[Section~2]{BZ16} establishes the following lemma, which is stronger than the statement of \cite[Lemma~1.1]{BZ16}.

\begin{lemma}\label{l:brick_def}
Let $(\Sigma^+, \sigma )$ be the full shift on binary sequences.
Suppose that $\ba=\{a_{n}\}_{n\in\N_0}$ is a strictly decreasing sequence of positive numbers converging to zero satisfying $\sum_{k\geq n}a_{k}=O(a_{n})$ as $n\to +\infty$. Let $\bb = \{b_{\bw}\}_{\bw \in \Sigma_*}$ be a weakly $\ba$-admissible gauge. For a family $\be =\{ e_{\bw} \}_{\bw \in \Sigma_*}$ of real numbers satisfying $\abs{e_{\bw}} \leq b_{\bw}$ for each $\bw \in \Sigma_*$, the Haar series $\sum_{\bw \in \Sigma_*} e_{\bw} h_{\bw}$ converges in the uniform norm to a function $f_{\be}$ in $\Lip(\Sigma^{+},d_{\ba})$. Moreover, the set
\begin{equation}  \label{e:Def_Hb}
\cH_{\bb} \= \{ f_{\be}  :  \be =\{ e_{\bw} \}_{\bw \in \Sigma_*} \text{ with }  \abs{e_{\bw}} \leq b_{\bw} \text{ for each }\bw \in \Sigma_* \}
\end{equation}
is a compact subset of $\Lip(\Sigma^{+},d_{\ba})$.
\end{lemma}


The set $\cH_{\bb}$ as in the lemma above is called the \emph{Hilbert brick} with gauge $\bb$. By means of Haar coefficients, it can be identified with the product space $\prod_{\bw \in \Sigma_*} [-b_{\bw}, b_{\bw}]$. In particular, we can endow each interval $[-b_{\bw}, b_{\bw}]$, $\bw \in \Sigma_*$, with a probability measure $\mu_{\bw}$. Then by using the identification from Lemma~\ref{l:brick_def}, we obtain a probability measure
$\P_{\bb} \= \prod_{\bw\in\Sigma_*}\mu_{\bw}$ on the Banach space $\Lip(\Sigma^{+},d_{\ba})$ supported on the compact subset $\cH_{\bb}$.


In this paper, we are most interested in the special case when all $\mu_{\bw}$ are uniform measures on $[-b_{\bw}, b_{\bw}]$.

\section{A concrete formulation of Theorem~\ref{t:prevalence}}     \label{s:concrete_formulation}
Theorem~\ref{t:prevalence} follows from an explicit construction of a compactly supported probability measure that is transverse to the complement of $\Lock_{\sigma} (\Sigma^{+},d_{\ba})$. Therefore, in order to establish Theorem~\ref{t:prevalence}, we prove Theorem~\ref{t:brick} below, which has concrete information on such a measure.

\begin{theorem}  \label{t:brick}
Let $(\Sigma^+, \sigma )$ be the full shift on binary sequences.
Let $\ba=\{a_{n}\}_{n\in\N_0}$ be a strictly decreasing sequence of positive numbers satisfying $a_{n+1}/a_n = O ( \theta^n)$ as $n \to +\infty$ for some $\theta \in (0,1/4)$. Set $b_n \= a_n / n$ for each $n\in\N_0$. Consider the weakly $\ba$-admissible gauge $\bb = \{ b_{\bw} \}_{\bw \in \Sigma_*}$ given by $b_{\bw} \= b_{\abs{\bw}}$ for each $\bw \in \Sigma_*$. Then there exists a probability measure $\P_{\bb}$ supported on the compact set $\cH_{\bb}$ such that for each $f_0 \in \Lip(\Sigma^{+},d_{\ba})$,
\begin{equation*}
\P_{\bb} (  \{ g \in \cH_{\bb}  :   f_0 + g \in \Lock_{\sigma} (\Sigma^{+},d_{\ba}) \} ) = 1  ,
\end{equation*}
and consequently, $\P_{\bb} ( \{ g \in \Lip(\Sigma^{+},d_{\ba})  :   f_0 + g \notin \Lock_{\sigma} (\Sigma^{+},d_{\ba}) \} ) = 0$.
\end{theorem}

A proof of Theorem~\ref{t:brick} will be presented in Section~\ref{s:Proof_T_Brick}.

\begin{proof}[Proof of Theorem~\ref{t:prevalence} assuming Theorem~\ref{t:brick}]
Theorem~\ref{t:brick} concludes that the measure $\P_{\bb}$ is transverse (as defined in Section~\ref{s:Introduction}) to the complement of the set $\Lock_{\sigma} (\Sigma^{+},d_{\ba})$ and is supported on the compact set $\cH_{\bb}$. Thus $\Lock_{\sigma}(\Sigma^{+},d_{\ba})$ is a prevalent set (as defined in Section~\ref{s:Introduction}) in $\Lip (\Sigma^{+},d_{\ba})$. Therefore Theorem~\ref{t:brick} implies Theorem~\ref{t:prevalence}.
\end{proof}

\section{A gap criterion for the locking property}   \label{s:Graph_gap}
From now on, our aim is to establish Theorem~\ref{t:brick}. To this end, we need to utilize a criterion for the locking property in terms of some probabilistic objects for the maximum mean cycle problem on de~Bruijn--Good digraphs. 

Let us make some remarks to explain the relations between (infinite dimensional) ergodic optimization, finite-dimensional ergodic optimization, and the maximum mean cycle problem on digraphs. The facts stated below until the start of Subsection~\ref{ss:Graph_theory} serve merely to assist the reader in comprehending our approach and will not be used in the proofs.

Recall that $P_{\sigma}$ denotes the set of all periodic points of $(\Sigma^{+},\sigma)$. Each periodic orbit in $P_{\sigma}$ is associated to a unique invariant probability measure supported on it. It is well-known that the set of these measures is a dense subset in $\MMM(\Sigma^{+},\sigma)$. Consider a sequence of maps $\{\pi_{n}\}_{n\in \N_{0}}$ where $\pi_{n} \:\MMM(\Sigma^{+},\sigma) \rightarrow \R^{2^{n}}$ is given by $\pi_{n}(\mu) \=\{ \mu ( [\bw] ) \}_{\bw\in\Sigma_{n}}$ for $\mu \in \MMM(\Sigma^{+},\sigma)$. Denote $R_{n} \= \pi_{n}(\cM(\Sigma^{+},\sigma))$. It then follows from Sections~3 and~4 in \cite{Zi95} (see also Subsection~3.5 in \cite{BZ16}) that $R_{0}\subset R_{1}\subset \cdots$ form a nested sequence of finite-dimensional polyhedra, whose vertices are measures supported on some periodic orbits in $P_{\sigma}$. Moreover, each polyhedron $R_{n}$ is a projection of the next one $R_{n+1}$, and $\MMM(\Sigma^{+},\sigma)$ can be recovered as the inverse limit of the sequence $\{R_{n}\}_{n\in\N_0}$.

Since $\MMM(\Sigma^{+},\sigma)$ is a Poulsen simplex, the polyhedra in $\{R_{n}\}_{n\in\N_0}$ are not regular simplices.  On the contrary, they have huge numbers of vertices, and their faces are small. By Proposition~3.5 in \cite{BZ16}, $R_{n}$ is isomorphic to the so-called ``circulation polytope'' of the de~Bruijn--Good digraph $\BG_n$. With the aid of this isomorphism, the finite-dimensional ergodic optimization problem over $R_{n}$ can be conveniently restated as the maximum mean cycle problem on the de~Bruijn--Good digraph $\BG_n$, and a criterion for the locking property is expressed in terms of the gap between the mean-weights of two heaviest directed cycles in the de~Bruijn--Good digraph accordingly, see the gap criterion in Subsection~\ref{ss:gapcriterion}. This step requires some notions from graph theory.

\subsection{Graph theory} \label{ss:Graph_theory}
In this subsection, we go over some key concepts and notation from graph theory. For a more detailed introduction to graph theory, we refer to \cite{BM08}.

A \emph{directed graph} (or \emph{digraph}) $\cD$ is an ordered pair $(V(\cD),A(\cD))$ of sets, where $V(\cD)$ is the set of \emph{vertices} and $A(\cD)$ is the set of \emph{arcs}, together with an \emph{incidence function} $\psi_\cD \: A(\cD) \rightarrow V(\cD) \times V(\cD)$ that sends an arc $a$ to an ordered pair $(u,v)$ of vertices. For an arc $a\in A(\cD)$ with $\psi_\cD(a) = (u,v)$, we call the vertex $u$ the \emph{tail} of $a$, denoted by $\tail(a)$, and call $v$ the \emph{head} of $a$, denoted by $\head(a)$. 

A digraph $\cD'$ is a \emph{subdigraph} of a digraph $\cD$ if $V(\cD') \subseteq V(\cD)$, $A(\cD') \subseteq A(\cD)$, and $\psi_{\cD'}$ is the restriction of $\psi_\cD$ to $A(\cD')$. We then say that $\cD$ \emph{contains} $\cD'$ or that $\cD'$ is \emph{contained in} $\cD$, and write $\cD\supseteq \cD'$ or $\cD' \subseteq \cD$, respectively.

A \emph{directed cycle} $\cC$ is a digraph satisfying that $\card{A(\cC)} = \card {V(\cC)} \eqqcolon n  \in \N$, that $V(\cC) = \bigcup_{a \in A(\cC)} \{ \tail(a), \, \head(a) \} $,  and that there is an enumeration $a_0, \, a_1,\, \dots, \, a_{n-1}$ of $A(\cC)$ such that $\head(a_i) = \tail(a_{i+1})$ for each $i\in\{0,\,1,\,\dots, \, n-1\}$, where $a_n \= a_0$. For a directed cycle $\cC$ we write $\len(\cC)$ for the length (i.e., the number of arcs) of $\cC$. For a digraph $D$, we denote by $\sC(D)$ the collection of all directed cycles contained in $D$.

The de~Bruijn--Good digraph $\BG_n$, $n\in\N$, is a digraph whose set of vertices $V(\BG_n) \= \Sigma_{n-1}$ and set of arcs $A(\BG_n) \= \Sigma_n$ consists of words of length $n-1$ and words of length $n$, respectively, which satisfies that $\tail(w_1 \dots w_n) = w_1\dots w_{n-1}$ and $\head(w_1 \dots w_n) = w_2 \dots w_{n}$ for each $n$-word $w_1 \dots w_n \in A(\BG_n)$. The first five de~Bruijn--Good digraphs are shown in Figure~\ref{f:BG}.

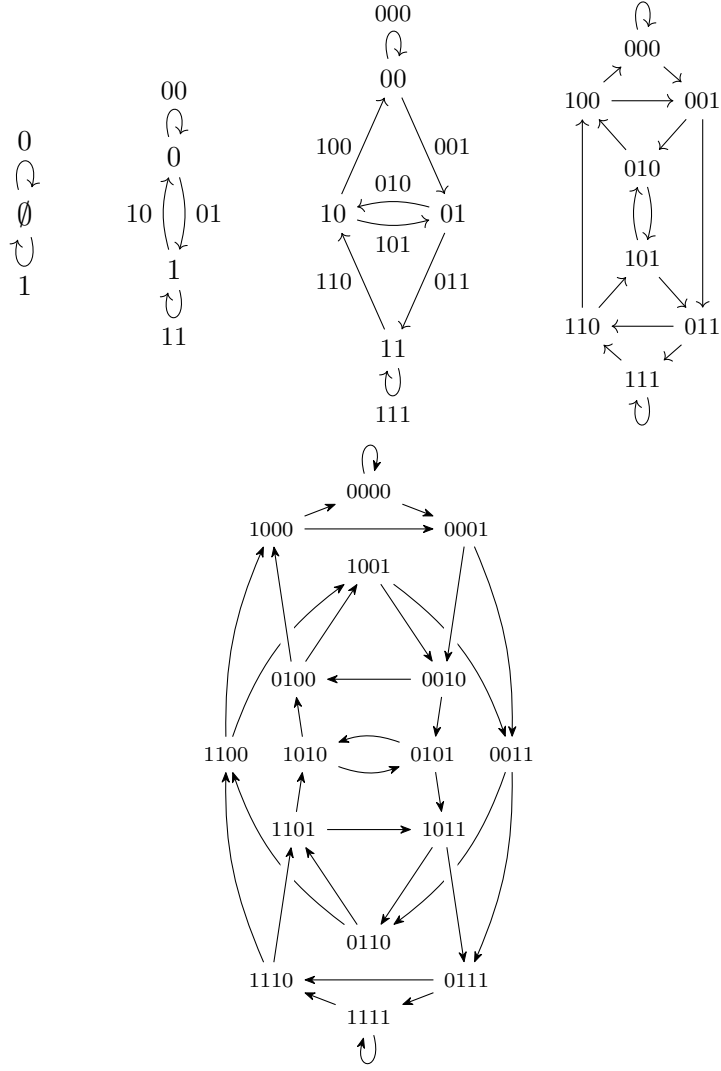
\begin{figure}[htb]
\begin{tikzpicture}[baseline] 
   \node(emptyword) at (0, 0){$\emptyset$};
   \draw(emptyword)edge[loop above]  node{\small $0$} (emptyword);
   \draw(emptyword)edge[loop below] node{\small $1$} (emptyword);
\end{tikzpicture}
\qquad
\begin{tikzpicture}[baseline] 
   \node(0) at (0,.75){\small $0$};
   \node(1) at (0,-.75){\small $1$};
   \draw(0)edge[loop above] node{\footnotesize $00$} (0);
   \draw(0)edge[->,bend left=15] node[right]{\footnotesize $01$} (1);
   \draw(1)edge[loop below] node{\footnotesize $11$} (1);
   \draw(1)edge[->,bend left=15] node[left]{\footnotesize $10$} (0);
\end{tikzpicture}
\qquad
\begin{tikzpicture}[baseline] 
   \node(00) at (  0, 1.8){\footnotesize $00$};
   \node(01) at ( .8,   0){\footnotesize $01$};
   \node(10) at (-.8,   0){\footnotesize $10$};
   \node(11) at (  0,-1.8){\footnotesize $11$};
   \draw(00)edge[loop above]       node{\scriptsize $000$} (00);
   \draw(00)edge[->]               node[right]{\scriptsize $001$} (01);
   \draw(01)edge[->,bend right=15] node[above]{\scriptsize $010$} (10);
   \draw(01)edge[->]               node[right]{\scriptsize $011$} (11);
   \draw(10)edge[->]               node[left] {\scriptsize $100$} (00);
   \draw(10)edge[->,bend right=15] node[below]{\scriptsize $101$} (01);
   \draw(11)edge[->]               node[left] {\scriptsize $110$} (10);
   \draw(11)edge[loop below]       node{\scriptsize $111$} (11);
\end{tikzpicture}
\qquad
\begin{tikzpicture}[baseline] 
   \node(000) at (  0, 2.2){\scriptsize $000$};
   \node(001) at ( .8, 1.5){\scriptsize $001$};
   \node(100) at (-.8, 1.5){\scriptsize $100$};
   \node(010) at (  0,  .6){\scriptsize $010$};
   \node(101) at (  0, -.6){\scriptsize $101$};
   \node(011) at ( .8,-1.5){\scriptsize $011$};
   \node(110) at (-.8,-1.5){\scriptsize $110$};
   \node(111) at (  0,-2.2){\scriptsize $111$};
   \draw(000)edge[loop above]      (000);
   \draw(000)edge[->]              (001);
   \draw(001)edge[->]              (010);
   \draw(001)edge[->]              (011);
   \draw(010)edge[->]              (100);
   \draw(010)edge[->,bend left=15] (101);
   \draw(011)edge[->]              (110);
   \draw(011)edge[->]              (111);
   \draw(100)edge[->]              (000);
   \draw(100)edge[->]              (001);
   \draw(101)edge[->,bend left=15] (010);
   \draw(101)edge[->]              (011);
   \draw(110)edge[->]              (100);
   \draw(110)edge[->]              (101);
   \draw(111)edge[->]              (110);
   \draw(111)edge[loop below]      (111);
\end{tikzpicture}
 \quad
 \begin{tikzpicture}[baseline,>={Stealth[round]}]
   \node(0000) at (   0, 3.5){\tiny $0000$};
   \node(0001) at ( 1.3, 3  ){\tiny $0001$};
   \node(1000) at (-1.3, 3  ){\tiny $1000$};
   \node(1001) at (   0, 2.5){\tiny $1001$};
   \node(0010) at (   1, 1  ){\tiny $0010$};
   \node(0100) at (  -1, 1  ){\tiny $0100$};
   \node(0011) at ( 1.9, 0  ){\tiny $0011$};
   \node(0101) at ( .85, 0  ){\tiny $0101$};
   \node(1010) at (-.85, 0  ){\tiny $1010$};
   \node(1100) at (-1.9, 0  ){\tiny $1100$};
   \node(1011) at (   1,-1  ){\tiny $1011$};
   \node(1101) at (  -1,-1  ){\tiny $1101$};
   \node(0110) at (  0 ,-2.5){\tiny $0110$};
   \node(0111) at ( 1.3,-3  ){\tiny $0111$};
   \node(1110) at (-1.3,-3  ){\tiny $1110$};
   \node(1111) at ( 0  ,-3.5){\tiny $1111$};

   \draw(0110)edge[->,bend left=17] (1100);
   \draw(1001)edge[->,bend left=17] (0011);
   \draw(1100)edge[->,bend left=17] (1001);
   \draw(0011)edge[->,bend left=17] (0110);

   \draw(0001)edge[line width=6pt,draw=white] (0010);
   \draw(0100)edge[line width=6pt,draw=white] (1000);
   \draw(1011)edge[line width=6pt,draw=white] (0111);
   \draw(1110)edge[line width=6pt,draw=white] (1101);

   \draw(0000)edge[loop above]      (0000);
   \draw(0000)edge[->]              (0001);
   \draw(0001)edge[->]              (0010); 
   \draw(0001)edge[->,bend left=12] (0011);
   \draw(0010)edge[->]              (0100);
   \draw(0010)edge[->]              (0101);
   \draw(0011)edge[->,bend left=12]  (0111);
   \draw(0100)edge[->]              (1000); 
   \draw(0100)edge[->]              (1001);
   \draw(0101)edge[->,bend right=20](1010);
   \draw(0101)edge[->]              (1011);
   \draw(0110)edge[->]              (1101);
   \draw(0111)edge[->]              (1110);
   \draw(0111)edge[->]              (1111);
   \draw(1000)edge[->]              (0000);
   \draw(1000)edge[->]              (0001);
   \draw(1001)edge[->]              (0010);
   \draw(1010)edge[->]              (0100);
   \draw(1010)edge[->,bend right=20](0101);
   \draw(1011)edge[->]              (0110);
   \draw(1011)edge[->]              (0111); 
   \draw(1100)edge[->,bend left=12] (1000);
   \draw(1101)edge[->]              (1010);
   \draw(1101)edge[->]              (1011);
   \draw(1110)edge[->,bend left=12] (1100);
   \draw(1110)edge[->]              (1101); 
   \draw(1111)edge[->]              (1110);
   \draw(1111)edge[loop below]      (1111);
 \end{tikzpicture}
\caption{The de~Bruijn--Good digraphs $\BG_n$ for $1 \le n \le 5$.}\label{f:BG}
\end{figure}

\subsection{Gap criterion}\label{ss:gapcriterion}

Fix a real-valued continuous function $f \in \CCC (\Sigma^+)$ and an integer $n \in \N$. We define inductively weights $W_n^f(\bw)$ associated to $f$ on the arcs $\bw$ of the de~Bruijn--Good digraph $\BG_n$ as follows:
\begin{itemize}
\smallskip
  \item If $n=1$, then $W_{1}^f(\alpha ) \=\frac{(-1)^{\alpha}}{2} c_{\emptyset}(f)$ for each $\alpha\in\{0,1\}$;

\smallskip
  \item If $n\geq 2$, then $W_{n}^f (\bw\alpha )\=W_{n-1}^f(\bw )+ \frac{(-1)^{\alpha}}{2}(c_{\bw}(f))$ for all $\bw\in\Sigma_{n-1}$ and $\alpha\in\{0,1\}$.
\end{itemize}
Thus
\begin{equation}\label{e:def_Wn}
W_n^f(\underbrace{x_{1} x_{2}\dots x_{n}}_{\bw} )  =  \frac12 \sum_{i=0}^{n-1}(-1)^{x_{i+1} } c_{x_{1} x_{2} \dots x_{i}}(f).
\end{equation}

The pair $\bigl( \BG_{n}, W_n^f \bigr)$ induces the \emph{mean-weight} (induced by $f$)
\begin{equation}  \label{e:def_MnC}
M_n^f(\cC)  \= \frac{1}{\len(\cC)} \sum_{ \bw \in A(\cC) } W_n^f (\bw)
\end{equation}
on each directed cycle $\cC$ contained in $\BG_n$.

Denote by $\cC_1$ and $\cC_2$ the two heaviest (in terms of mean-weight, in descending order) directed cycles contained in $\BG_n$. Then define
\begin{equation*}
	\Gap_n(f) \= M_n^f(\cC_1) - M_n^f(\cC_2) \geq 0 
\end{equation*}
to be the gap in the mean-weights of $\cC_1$ and $\cC_2$ in $\BG_n$.

It is easy to check that 
\begin{align}  
W_n^f ( \bw ) & =  W_n^{ A_n f } ( \bw ), \label{e:W_f_vs_Anf} \\
M_n^f ( \cC )  & =  M_n^{ A_n f } ( \cC ),  \label{e:M_f_vs_Anf} \\
\Gap_n(f)       & =  \Gap_n(A_n f) \label{e:Gap_f_vs_Anf}  
\end{align}
for $\bw \in \Sigma_n$ and $\cC \in A(\BG_n)$.

In the notations above, we have the following gap criterion from \cite[Lemma 4.2]{BZ16}.

\begin{lemma}[Gap criterion for the locking property \cite{BZ16}]\label{l:gap}
Let $(\Sigma^+, \sigma )$ be the full shift on binary sequences.
Consider a strictly decreasing sequence $\ba = \{a_n\}_{n\in\N_0}$ of positive numbers satisfying
\begin{equation}\label{e:summability}
\sum_{n=1}^{+\infty} n a_n < +\infty   .
\end{equation}
If a function $f \in \Lip(\Sigma^{+},d_{\ba})$ satisfies the following inequality:
\begin{equation}\label{e:gap}
\Gap_n( f) > \sum_{k=n}^{+\infty} (k-n+1) \max_{ \bw \in \Sigma_k } \{ \abs{c_{\bw}(f)} \}  .
\end{equation}
for some $n \in \N$, then $f \in \Lock_{\sigma} (\Sigma^{+},d_{\ba})$.
\end{lemma}

In informal terms, if the tail of the Haar series is small compared to the gap of its initial part (up to the length $n-1$ (see \eqref{e:A_truncated}, \eqref{e:def_Wn}, and \eqref{e:Gap_f_vs_Anf})), so it does not influence the maximizing measure. Note that $f \in \Lock_{\sigma} (\Sigma^{+},d_{\ba})$ then in particular $f  \in \sP\sU(\sigma)$.

\section{Random perturbations on the arc-weight}   \label{s:random_perturbation}
In this section, we will describe the random weights on arcs of the de~Bruijn--Good digraphs $\BG_{n}$, $n\in\N$. Namely, we deal with the Haar expansion of $f_0 +g$, with $f_0 \in \Lip(\Sigma^{+},d_{\ba})$ (the deterministic part) and $g\in\mathcal{H}_{\bb}$ (the random part).

Let $\{Y_{\bw}\}_{\bw\in\Sigma_*}$ be a collection of i.i.d.\ random variables, having the uniform distribution on $[-1, 1]$, indexed by the set $\Sigma_*$ of finite words.

Suppose that $\ba=\{a_{n}\}_{n\in\N_0}$ is a strictly decreasing sequence of positive numbers converging to zero satisfying $\sum_{k\geq n}a_{k}=O(a_{n})$ as $n\to +\infty$. Let $\{b_n\}_{n\in\N_0}$ be a sequence of positive numbers satisfying $b_n = o (a_n)$ as $n \to +\infty$.

We formally define a random function
\begin{equation}\label{e:random_function}
  g=\sum_{m= 0}^{+\infty} \sum_{\bw \in \Sigma_m} b_{m} \cdot Y_{\bw}\cdot h_{\bw},
\end{equation}
which has random coefficients on its Haar series. By Lemma~\ref{l:brick_def}, the series above always converge uniformly, and the limit $g$ is Lipschitz with respect to $d_{\ba}$. For each $n \in \N$, the $n$-th approximation $A_n g$ of $g$ is given by
\begin{equation}\label{e:truncation}
  A_{n}g =\sum_{m=0}^{n-1}\sum_{\bw\in\Sigma_m}b_{m}Y_{\bw}h_{\bw}.
\end{equation}

Fix an arbitrary function $f_0 \in \Lip(\Sigma^{+},d_{\ba})$. With abuse of notation, we use similar symbols for the random arc-weight, mean-weight, and gap induced by random functions that are continuous. In particular, the (random) arc-weights $W_n^{f_0+g}(\bw)$, $n\in\N$, on $\BG_{n}$ induced by $f_0 + g$ are defined by the following inductive process:
\begin{itemize}
\smallskip
  \item If $n=1$, then $W_{1}^{f_0+g}(\alpha) =\frac{(-1)^{\alpha}}{2} (c_{\emptyset}(f_0)+b_0 \cdot Y_{\emptyset})$ for each $\alpha\in\{0,1\}$;

\smallskip
  \item If $n\geq 2$, then $W_{n}^{f_0+g}(\bw\alpha) = W_{n-1}^{f_0+g}(\bw)+ \frac{(-1)^{\alpha}}{2}(c_{w_{1}w_{2}\cdots w_{n-1}}(f_0)+b_{n-1}\cdot Y_{\bw})$ for all $\bw\in\Sigma_{n-1}$ and $\alpha\in\{0,1\}$.
\end{itemize}

Fix an integer $n\in\N$. We then have
\begin{equation}\label{e:random_Wn}
W_n^{f_0+g}(\underbrace{w_{1} w_{2}\dots w_{n}}_{\bw})=\frac12 \sum_{i=0}^{n-1}(-1)^{w_{i+1} }(c_{w_{1}w_{2}\dots w_{i}}(f_0)+b_{i}Y_{w_{1} w_{2} \dots w_{i}}),
\end{equation}
for $\bw \in \Sigma_n$. Moreover, the (random) mean-weight of a directed cycle $\cC$ in $\BG_n$ induced by $f_0 +g$ is then given by
\begin{equation}  \label{e:random_MnC}
M_n^{f_0+g}(\cC) = \frac{1}{\len(\cC)} \sum_{ \bw \in A(\cC) } W_n^{f_0+g}( \bw ).
\end{equation}
Finally, the (random) gap, induced by $f_0 + g$, between the mean-weights of the two heaviest  (in terms of mean-weight) directed cycles $\cC_1$ and $\cC_2$ in $\BG_n$ is 
\begin{equation} \label{e:random_Gap}
\Gap_n (f_0+g) = M_n^{f_0+g} (\cC_1) - M_n^{f_0+g} (\cC_2) \geq 0.
\end{equation}
Since uniform measures are continuous, almost surely there is a unique directed cycle with the maximal mean-weight, i.e., $\Gap_n (f_0+g) > 0$.

The main technical result is formulated below.

\begin{prop}   \label{p:conditional_layered_Gap_estimate}
Let $(\Sigma^+, \sigma )$ be the full shift on binary sequences.
Let $\ba=\{a_{n}\}_{n\in\N_0}$ be a strictly decreasing sequence of positive numbers satisfying $a_{n+1}/a_n = O ( \theta^n)$ as $n \to +\infty$ for some $\theta \in (0,1/2)$. Let $\{b_n\}_{n\in\N_0}$ be a sequence of positive numbers satisfying $b_n = o (a_n)$ as $n \to +\infty$. Consider a weakly $\ba$-admissible gauge $\bb = \{ b_{\bw} \}_{\bw \in \Sigma_*}$ satisfying $b_{\bw}=b_{\abs{\bw}}$ for all $\bw \in \Sigma_*$. Fix an arbitrary function $f_0 \in \Lip(\Sigma^+, d_{\ba})$. Consider the random function $g$ defined in \eqref{e:random_function}.

Consider an integer $n \geq 2$, a real number $\epsilon \in (0, 1/2 )$,  and a collection $\bm{t} = \{ t_{\bw} \}_{ {\bw} \in \Sigma_{\leq n-2 }  }$ of real numbers $t_{\bw} \in\R$. Let $\{ X_{\bm{t}} (\cC) : \cC \in \sC (\BG_n) \}$ be a collection of random variables with joint distribution given by the conditional distribution of $\bigl\{ M_n^{f_0 +g}(\cC) : \cC \in \sC (\BG_n) \bigr\}$ given $\{Y_{\bw} = t_{\bw} : \bw \in \Sigma_{\leq n-2} \} $. Let $\Gap_{n,\bm{t}}$ be the absolute value of the difference between the largest two values in $\{ X_{\bm{t}} (\cC)  :  \cC \in \sC(\BG_n) \}$. Then
\begin{equation} \label{e:conditional_layered_Gap_estimate}
\P \bigl( \bigl\{ \Gap_{n, \bm{t}} \leq 2^{-n} b_{n-1} \epsilon  \bigr\} \bigr)
\leq   2^n \epsilon.
\end{equation}

\end{prop}

\begin{remark}
Note that $A(\cC_1) \nsubseteq A(\cC_2)$ for each pair of distinct directed cycles contained in a directed graph $\cD$. Then by \eqref{e:random_Wn}, \eqref{e:random_MnC}, and the definition of $Y_{\bw}$ for $\bw\in\Sigma_*$, we get that the random functions $X_{\bm{t}} (\cC)$ and $X_{\bm{t}} (\cC')$ are distinct linear combinations of random variables $Y_{\bw}$, $\bw\in\Sigma_{\leq n-1}$, (with scalar multiples of Haar functions as coefficients) for distinct directed cycles $\cC$ and $\cC'$ contained in $\BG_n$. Moreover, $\Gap_{n, \bm{t}}$ in the proposition is well-defined almost surely.
\end{remark}

	The strategy of our proof of Proposition~\ref{p:conditional_layered_Gap_estimate}  is inspired by the proof of Theorem~3 in \cite{CCK15}. The proof in \cite{CCK15} states the anti-concentration inequalities of Gaussian random vectors and the proof lies in bounding the probability density function of the maximum of a Gaussian random vector. Our proof has three major differences:
	\begin{itemize}
		\smallskip
		\item[(i)] Due to the definition of prevalence, it requires a probability measure which is compactly supported. Thus we cannot directly use the jointly Gaussian random variables (which are not compactly supported) and the resulting estimates in the proof of Theorem~3 in \cite{CCK15}.
				
		\smallskip
		
		\item[(ii)] While in \cite{CCK15} special properties of jointly Gaussian random variables are used in an essential way in the general case, we have to work with uniform distributions relying on the combinatorics of de~Bruijn--Good digraphs and specific edge weights.
		
		 \smallskip
		
		\item[(iii)] Our strategy concentrates on the upper bound on the probability of small conditional gap $\Gap_{n,\bm{t}}$, rather than the maximum of a Gaussian random vector in the proof of Theorem~3 in \cite{CCK15}.
	\end{itemize}

\begin{proof}[Proof of Proposition~\ref{p:conditional_layered_Gap_estimate}]
Consider $n$, $\epsilon$, and $\bm{t}$ as given in the hypothesis.

Let $\cE$ be the event that $\abs{ Y_{\bw} } \leq 1 - 2 \epsilon$ for all $\bw \in \Sigma_{n-1}$.

We will investigate, for each $\cC\in \sC(\BG_n)$, the standard deviation $\sigma_{\cC} \= \sigma (X_{\bm{t}}(\cC))$, the mean $\mu_{\cC} \= \E[ X_{\bm{t}} (\cC) ]$, the probability density function $\phi_{\cC}$ for $X_{\bm{t}} (\cC)$, and the probability density function $\tphi_{\cC}$ for $X_{\bm{t}} (\cC)$ restricted to $\cE$. Then it follows from  the definition of $X_{\bm{t}}$, the fact that $\E [Y_{\bw} ] = 0$ for $\bw\in\Sigma_*$, \eqref{e:random_Wn}, and \eqref{e:random_MnC} that
\begin{equation}  \label{e:Pf_p_conditional_layered_Gap_estimate_Xt}
X_{\bm{t}}(\cC) = \mu_{\cC} + \frac{b_{n-1}}{2 \len(\cC) } \sum_{\bw \in V(\cC)} (-1)^{z_{\cC}(\bw)} Y_{\bw}
\end{equation}
for each $\cC\in \sC(\BG_n)$. Here and henceforth by $z_{\cC}(\bw)$ we denote, for each $\bw \in V(\cC)$, the unique number in $\{0, \, 1\}$ such that $\bw z_{\cC}(\bw) \in A(\cC)$. Note that 
\begin{equation}  \label{e:Pf_p_conditional_layered_Gap_estimate_Xt_bound}
\abs{ X_{\bm{t}} (\cC) - \mu_{\cC} } \leq b_{n-1} /2.
\end{equation}

Define, for each $\cC \in \sC(\BG_n)$ and each pair of numbers $x,\,x'\in\R$, the event
\begin{equation*}
\cR_{\bm{t}}(\cC,x') \= \{ X_{\bm{t}} (\cC')  \leq x' \text{ for all } \cC' \in \sC(\BG_n) \setminus \{ \cC \} \},
\end{equation*} 
and the conditional probabilities
\begin{align*}
Q_{\cC}(x, \, x')   & \= \P( \cR_{\bm{t}}(\cC,x')  \,|\, X_{\bm{t}} ( \cC ) = x  ),  \\
\tQ_{\cC}(x, \, x') & \= \P( \cE \cap \cR_{\bm{t}}(\cC,x')  \,|\, X_{\bm{t}} ( \cC ) = x ),   \\
Q_{\cC}(x)           & \= Q_{\cC}(x, \, x), \\
\tQ_{\cC}(x)         & \= \tQ_{\cC}(x, \, x).
\end{align*}

We first verify the following two claims.

\smallskip

\emph{Claim~1.} $\phi_{\cC} (x + \Delta) \geq \tphi_{\cC} (x)$ for all $\cC \in \sC(\BG_n)$, $\Delta \in [0, b_{n-1} \epsilon ]$, and $x \in \R$ satisfying $\abs{x - \mu_{\cC} } \leq  1 - b_{n-1} \epsilon$.

\smallskip
To establish Claim~1, we fix $\cC$, $\Delta$, and $x$ as in the claim. We consider a collection $\{ y_{\bw} \}_{\bw \in V(\cC) }$ of real numbers satisfying both
\begin{equation}  \label{e:Pf_p_conditional_layered_Gap_estimate_claim}
\mu_{\cC} + \frac{b_{n-1}}{2 \len(\cC) } \sum_{\bw \in V(\cC)} (-1)^{z_{\cC}(\bw)} y_{\bw} = x
\end{equation}
and $\abs{y_{\bw} } \leq 1- 2 \epsilon$ for each $\bw \in V(\cC)$. We set
\begin{equation*}
y'_{\bw} \= y_{\bw} + (-1)^{z_{\cC}(\bw)}  2\Delta / b_{n-1}
\end{equation*}
for each $\bw \in V(\cC)$. Then
\begin{equation*}
\abs{y'_{\bw} } 
\leq 1 - 2 \epsilon + \frac{2\Delta}{b_{n-1}}
\leq 1
\end{equation*}
for each $\bw \in V(\cC)$, and
\begin{align*}
&\mu_{\cC} + \frac{b_{n-1}}{2 \len(\cC) } \sum_{\bw \in V(\cC)} (-1)^{z_{\cC}(\bw)} y'_{\bw} \\
&\qquad = \mu_{\cC} + \frac{b_{n-1}}{2 \len(\cC) } \biggl(  \sum_{\bw \in V(\cC)} (-1)^{z_{\cC}(\bw)} y_{\bw}  + \sum_{\bw \in V(\cC)} \frac{2\Delta}{b_{n-1}} \biggr) \\
&\qquad = x + \Delta .
\end{align*}
Claim~1 then follows from \eqref{e:Pf_p_conditional_layered_Gap_estimate_Xt} and the translation invariance of uniformly distributed random variables.

\smallskip

By a similar argument as in the proof of Claim~1, we get the following claim.

\smallskip

\emph{Claim~2.} $Q_{\cC} \bigl( x + \Delta, \, x + \bigl( 1 - 2^{-n +1}   \bigr) \Delta \bigr) \geq \tQ_{\cC} (x)$  for all $\cC \in \sC(\BG_n)$, $\Delta \in [0, b_{n-1} \epsilon ]$, and $x \in \R$ satisfying $\abs{x - \mu_{\cC} } \leq  1 - b_{n-1} \epsilon$.

\smallskip

To establish Claim~2, we fix $\cC$, $\Delta$, and $x$ as in the claim. We consider a collection $\{ y_{\bw} \}_{\bw \in \Sigma_{n-1} }$ of real numbers satisfying both the equality in \eqref{e:Pf_p_conditional_layered_Gap_estimate_claim} and $\abs{y_{\bw} } \leq 1- 2 \epsilon$ for each $\bw \in \Sigma_{n-1}$. We set
\begin{equation*}
y'_{\bw} \= y_{\bw} + (-1)^{z_{\cC}(\bw)}  2\Delta / b_{n-1}
\end{equation*}
for each $\bw \in V(\cC)$, and set $y'_{\bu} \= y_{\bu}$ for each $\bu \in \Sigma_{n-1} \setminus V(\cC)$. Then $\abs{y'_{\bw} } \leq 1$ for all $\bw \in \Sigma_{n-1}$ as
\begin{equation*}
\abs{y'_{\bw} } 
\leq 1 - 2 \epsilon + \frac{2\Delta}{b_{n-1}}
\leq 1
\end{equation*}
for each $\bw \in V(\cC)$. On the other hand, consider an arbitrary directed cycle $\cC' \in \sC(\BG_n) \setminus \{\cC\}$. Since $\cC'$ and $\cC$ are distinct directed cycles, it is clear that $A(\cC') \setminus A(\cC) \neq \emptyset$. Consequently,
\begin{align*}
&\mu_{\cC'} + \frac{b_{n-1}}{2 \len(\cC') } \sum_{\bw \in V(\cC')} (-1)^{z_{\cC'}(\bw)} y'_{\bw} \\
&\qquad \leq \mu_{\cC'} + \frac{b_{n-1}}{2 \len(\cC') } \biggl(  \sum_{\bw \in V(\cC')} (-1)^{z_{\cC'}(\bw)} y_{\bw}  - \frac{2\Delta}{b_{n-1}} + \sum_{\bw \in V(\cC')} \frac{2\Delta}{b_{n-1}} \biggr) \\
&\qquad = x + \frac{\Delta}{ \len(\cC') } (    - 1 + \card V(\cC') ) \\
&\qquad \leq x + \bigl( 1 - 2^{-n +1}   \bigr) \Delta.
\end{align*}
Claim~2 then follows from \eqref{e:Pf_p_conditional_layered_Gap_estimate_Xt} and the translation invariance of uniformly distributed random variables.

\smallskip

Finally, we apply the two claims to finish the proof of the proposition. 

Provided $x'<x$, we define the events $\cQ(\cC) \= \cR_{\bm{t}}( \cC, x' ) \cap \{ X_{\bm{t}}(\cC) = x \}$, $\cC \in \sC(\BG_n)$. We observe that $\cQ(\cC)$ and $\cQ(\cC')$ are disjoint for distinct directed cycles $\cC$ and $\cC'$ in $\sC(\BG_n)$.

Define $\iota \= b_{n-1} \epsilon$. Then by the two claims, the observation above, \eqref{e:Pf_p_conditional_layered_Gap_estimate_Xt}, \eqref{e:Pf_p_conditional_layered_Gap_estimate_Xt_bound}, and the definitions of $Q_{\cC}$, $\tQ_{\cC}$, $\cR_{\bm{t}}$, and $\cE$, we get
\begin{align*}  
&\P  \bigl(   \Gap_{n,\bm{t}} \geq 2^{ - n + 1 }\iota \bigr) \\
&\qquad\geq \sum_{\cC \in \sC(\BG_n)}  \int_{\iota + \mu_{\cC} -  2^{-1}  b_{n-1}  }^{ - \iota + \mu_{\cC} +   2^{-1}  b_{n-1}  } \!
         Q_{\cC} \bigl( x+\iota, \, x + \bigl(1 - 2^{-n + 1} \bigr) \iota \bigr) \phi_{\cC}( x + \iota ) \,\mathrm{d}x  \notag \\
&\qquad\geq \sum_{\cC \in \sC(\BG_n)}  \int_{\iota + \mu_{\cC} -  2^{-1}  b_{n-1}}^{- \iota + \mu_{\cC} +  2^{-1}  b_{n-1}  } \!
         \tQ_{\cC} (x) \tphi_{\cC}(x) \,\mathrm{d}x  \notag \\
&\qquad\geq \sum_{\cC \in \sC(\BG_n)}  \P( \cE  \cap \{ \cC \text{ is the cycle with maximal mean-weight} \} ) \\         
&\qquad = \P( \cE ).
\end{align*}
We remark that in the last inequality, we used the fact that $\tphi_{\cC}$ is zero outside of $[ \mu_{\cC} -  2^{-1}  b_{n-1},  \mu_{\cC} +  2^{-1}  b_{n-1}]$ by \eqref{e:Pf_p_conditional_layered_Gap_estimate_Xt_bound}.

As a straightforward consequence of the uniform distribution and the hypothesis that $\epsilon \in (0,1/2)$, we have
\begin{equation*}
\P(\cE) =  (1-2 \epsilon)^{2^{n-1}} \geq 1- 2^n \epsilon ,
\end{equation*}
where the inequality follows from the well-known inequality $1-m\alpha \leq (1-\alpha)^m$ for $m\in\N$ and $\alpha \in (0,1)$ from the concavity of the logarithm. The proposition is therefore established.
\end{proof}

The following proposition brings us a step closer to the gap criterion in Lemma~\ref{l:gap}.

\begin{prop}\label{p:goal}
Let $(\Sigma^+, \sigma )$ be the full shift on binary sequences.
Let $\ba=\{a_{n}\}_{n\in\N_0}$ be a strictly decreasing sequence of positive numbers satisfying $a_{n+1}/a_n = O ( \theta^n)$ as $n \to +\infty$ for some $\theta \in (0,1/2)$. Let $\{b_n\}_{n\in\N_0}$ be a sequence of positive numbers satisfying $b_n = o (a_n)$ as $n \to +\infty$. Consider a weakly $\ba$-admissible gauge $\bb = \{ b_{\bw} \}_{\bw \in \Sigma_*}$ satisfying $b_{\bw}=b_{\abs{\bw}}$ for all $\bw \in \Sigma_*$.

Then there exists a measure $\P_{\bb}$ supported on $\mathcal{H}_{\bb}$ with the property that for every $f_0\in \Lip(\Sigma^{+},d_{\ba})$, we have
\begin{equation*}\label{e:goal}
\mathbb{P}_{\bb} \bigl( \bigl\{ g\in\cH_{\bb} :  \exists  N\in\N , \, \forall n \geq N, \, \Gap_n(f_0+g) \geq 4^{-n} n^{-3} b_{n-1}  \bigr\} \bigr) =1.
\end{equation*}
\end{prop}

\begin{proof}
Fix an arbitrary integer $n\geq 2$. 

It follows from Proposition~\ref{p:conditional_layered_Gap_estimate} that
\begin{equation*}
\P ( \{ \Gap_{n, \bm{t}} \leq 2^{-n} b_{n-1} \epsilon  \} )
\leq 2^n \epsilon 
\end{equation*}
for all $\epsilon \in (0,1/2)$.

We denote by $\phi_{n-2}$ the joint probability density function for all the random variables in $\{ Y_{\bw}  :  \bw \in \Sigma_{\leq n-2} \}$. Then for each $\epsilon \in (0, 1/2)$,
\begin{align*}
&   \P\{\Gap_{n}(f_0 + g) \leq 2^{-n}b_{n-1}\epsilon\}  \\
&\qquad   \leq \int \! \P\{\Gap_{n,\bm{t}} \leq 2^{-n} b_{n-1} \epsilon\}\phi_{n-2} (\bm{t}) \, \mathrm{d} \bm{t} \\
&\qquad   \leq 2^n \epsilon.
\end{align*}


Taking $\epsilon \= 2^{-n} n^{-3}$, then , we have
\begin{equation*}
 \sum_{n=2}^{+\infty}\P\{\Gap_n (f_0 + g) \leq 4^{-n} n^{-3} b_{n-1}\} 
 \leq \sum_{n=2}^{+\infty}  n^{-3}  
< +\infty.
\end{equation*}
By the Borel--Cantelli lemma, therefore, the proposition follows.
\end{proof}

\section{Proof of Theorem~\ref{t:brick}}   \label{s:Proof_T_Brick}

We are now ready to establish the more concrete formulation Theorem~\ref{t:brick} of our main theorem (Theorem~\ref{t:prevalence}).

\begin{proof}[Proof of Theorem~\ref{t:brick}]
Fix an arbitrary function $f_0\in \Lip(\Sigma^{+},d_{\ba})$. Define
\begin{equation*}
\Lambda_{f_0} 
\= \bigl\{ g\in\cH_{\bb}  :  \exists  N\in\N , \, \forall n \geq N, \, \Gap_n(f_0+g) \geq 4^{-n} n^{-3} b_{n-1}  \bigr\} .
\end{equation*}

Next, for each $g\in\Lambda_{f_0}$, by the definition of $\Lambda_{f_0}$, for all sufficiently large $n \in \N$,
\begin{equation}\label{e:biggap}
\Gap_n(f_0+g) \geq  4^{-n} n^{-3} b_{n-1}  .
\end{equation}
On the other hand, since by the hypothesis on $\ba$ and $\bb$,
\begin{equation}\label{e:summable}
\sum_{k=1}^{ + \infty}ka_k <+\infty,
\end{equation}
and since $\theta \in (0,1/4)$, for all sufficiently large integers $n$, one has
\begin{equation}\label{e:ratio}
  \frac{b_{n-1}}{b_{n}} >  \max\{  \LIP_{d_{\ba}}(f_0), \, 1\} 4^{n+1}  n^4
\end{equation}
and
\begin{equation}\label{e:tail}
  a_n \geq \sum_{k=n+1}^{+\infty}(k-n+1) a_k .
\end{equation}

Combining \eqref{e:biggap}, \eqref{e:ratio}, $b_n = a_n / n$, \eqref{e:tail}, \eqref{e:c_vs_var}, \eqref{e:LIP_vs_var_over_an}, and the triangle inequality, we deduce that for each $N\in\N$, there exists an integer $n \geq N$ such that
\begin{align*}
&  \Gap_n(f_0 +g) \\
&\qquad \geq 4^{-n} n^{-3} b_{n-1}                                        & (\text{by }\eqref{e:biggap}) \\
&\qquad > 4 \max\{  \LIP_{d_{\ba}}(f_0 ), \, 1\} n  b_{n}       & (\text{by }\eqref{e:ratio}) \\
&\qquad = 4 \max\{  \LIP_{d_{\ba}}(f_0 ), \, 1\} a_n              & (\text{by }b_n = a_n / n) \\
&\qquad \geq 2 \max\{  \LIP_{d_{\ba}}(f_0 ), \, 1\} \sum_{k=n}^{+\infty}(k-n+1) a_k & (\text{by }\eqref{e:tail}) \\
&\qquad \geq \sum_{k=n}^{+\infty}(k-n+1) (  a_k \LIP_{d_{\ba}}(f_0 ) + b_k )   & (\text{by }b_k = a_k / k) .
\end{align*}
Note that by  \eqref{e:c_vs_var} and \eqref{e:LIP_vs_var_over_an}, $a_k \LIP_{d_{\ba}}(f_0 )  \geq  \max_{ \bw \in \Sigma_k } \{ \abs{ c_{\bw}(f_0 ) } \}$. On the other hand, since $g\in\Lambda_{f_0}$, by the definition of $\Lambda_{f_0}$ and Lemma~\ref{l:brick_def}, we have $b_k \geq \max_{ \bw \in \Sigma_k } \{ \abs{c_{\bw}(g)} \}$. Hence
\begin{align*}
&  \Gap_n(f_0 +g) \\
&\qquad \geq \sum_{k=n}^{+\infty}(k-n+1) \Bigl( \max_{ \bw \in \Sigma_k } \{ \abs{ c_{\bw}(f_0 ) } \}  +  \max_{ \bw \in \Sigma_k } \{ \abs{c_{\bw}(g)} \} \Bigr)  \\
&\qquad \geq \sum_{k=n}^{+\infty}(k-n+1)  \max_{ \bw \in \Sigma_k } \{ \abs{ c_{\bw}( f_0  + g ) } \}  .
\end{align*}
This means that $f_0 +g$ satisfies the gap criterion \eqref{e:gap} at length $n$. Together with \eqref{e:summable} and Lemma~\ref{l:gap} (on $f_0 +g$), it then follows that $f_0 +g \in \Lock_{\sigma} (\Sigma^{+},d_{\ba})$. 

Hence for each $f_0 \in \Lip(\Sigma^{+},d_{\ba})$ we have
\begin{equation*}
\Lambda_{f_0 }\subseteq \{g\in\mathcal{H}_{\bb}  :  f_0 +g \in \Lock_{\sigma} (\Sigma^{+},d_{\ba})\}.
\end{equation*}
Then it follows from Proposition~\ref{p:goal},
\begin{equation*}
1=\P_{\bb}(\Lambda_{f_0})\leq\P_{\bb}(\{g\in\cH_{\bb}  :  f_0 +g \in \Lock_{\sigma} (\Sigma^{+},d_{\ba})\})\leq 1.
\end{equation*}
The proof of Theorem~\ref{t:brick} is therefore complete.
\end{proof}

\section{Proof of Theorem~\ref{t:prevalence_lip}}      \label{s:gound_states}


Recall the definition of little Lipschitz spaces reviewed in Section~\ref{s:Introduction}.

\begin{proof}[Proof of Theorem~\ref{t:prevalence_lip}]
It suffices to show that $\lock_{\sigma}(\Sigma^{+},d_{\ba})$ is a prevalent set in $\lip(\Sigma^{+},d_{\ba})$.

Set $b_n \= a_n / n$ for each $n\in\N_0$. Consider the weakly $\ba$-admissible gauge $\bb = \{ b_{\bw} \}_{\bw \in \Sigma_*}$ given by $b_{\bw} \= b_{\abs{\bw}}$ for each $\bw \in \Sigma_*$.

We first show that the Hilbert brick $\cH_{\bb}$ defined in \eqref{e:Def_Hb} in Lemma~\ref{l:brick_def} is a subset of $\lip(\Sigma^{+},d_{\ba})$. 

Indeed, for each $\be =\{ e_{\bw} \}_{\bw \in \Sigma_*}$ with $\abs{e_{\bw}} \leq b_{\bw}$ for each $\bw \in \Sigma_*$, by Lemma~\ref{l:brick_def} the uniform limit $f_{\be} \in \Lip(\Sigma^{+},d_{\ba})$ of the Haar series $\sum_{\bw \in \Sigma_*} e_{\bw} h_{\bw}$ satisfies that for each $n \in \N$ sufficiently large and each pair of $\bx,\,\by \in \Sigma^{+}$ with $\bx  \disagree  \by = n$, by \eqref{e:Haar_function}, the hypothesis on $\ba$, and our choice of $b_n$, we have
\begin{equation*}
\abs{  f_{\be}(\bx) - f_{\be}(\by) } 
\leq \sum_{m=n}^{+\infty} b_m
\leq 2 b_n
= 2 a_n / n
= o( a_n ).
\end{equation*}
Thus by the definitions of the metric $d_{\ba}$ and of the little Lipschitz space, we have $f_{\be} \in \lip(\Sigma^{+},d_{\ba})$. It follows from \eqref{e:Def_Hb} that $\cH_{\bb} \subseteq \lip(\Sigma^{+},d_{\ba})$. 

Note that $\lip(\Sigma^{+},d_{\ba}) \subseteq \Lip(\Sigma^{+},d_{\ba})$ and that 
\begin{equation*}
\Lock_{\sigma}(\Sigma^{+},d_{\ba}) \cap \lip(\Sigma^{+},d_{\ba})  \subseteq \lock_{\sigma}(\Sigma^{+},d_{\ba}).
\end{equation*}
Hence by Theorem~\ref{t:brick}, the probability measure $\P_{\bb}$ from Theorem~\ref{t:brick} supported on the compact set $\cH_{\bb}$ satisfies that for each $ f_0 \in  \lip(\Sigma^{+},d_{\ba})$,
\begin{align*}
& \P_{\bb} (  \{ g \in \cH_{\bb}  :   f_0 + g \in \lock_{\sigma} (\Sigma^{+},d_{\ba}) \} ) \\
&\qquad  \geq \P_{\bb} (  \{ g \in \cH_{\bb}  :   f_0 + g \in \Lock_{\sigma}(\Sigma^{+},d_{\ba}) \cap \lip(\Sigma^{+},d_{\ba}) )  \\
&\qquad   =      \P_{\bb} (  \{ g \in \cH_{\bb}  :   f_0 + g \in \Lock_{\sigma}(\Sigma^{+},d_{\ba})  ) =1.
\end{align*}

Therefore, $\lock_{\sigma}(\Sigma^{+},d_{\ba})$ is a prevalent set in $\lip(\Sigma^{+},d_{\ba})$, and the theorem is established.
\end{proof}

As more detailed discussions in \cite[Section~1]{LZ23} suggest, there is a ``dictionary'' for the correspondences between thermodynamic formalism and its tropical counterpart, ergodic optimization. In this dictionary, the existence, uniqueness, and equidistribution of periodic points for the equilibrium state may be considered as the counterparts to the existence, uniqueness, and periodicity of (the support of) the maximizing measure. In the current paper, we investigated in yet another ``language'', namely, a probabilistic one through random maximum mean cycle problems on digraphs using probabilistic tools. Here directed cycles on digraphs translate to periodic orbits, and maximum mean cycles to maximizing measures. Similar analogies can be drawn between these three ``languages'' in the ``dictionary.'' It would be interesting to investigate further into such connections, especially the probabilistic aspects.

\end{document}